\documentclass[11pt]{amsart}
\usepackage{amsmath,amscd,amssymb}

\newtheorem{theorem}{Theorem}[section]
\newtheorem{lemma}[theorem]{Lemma}
\newtheorem{proposition}[theorem]{Proposition}
\newtheorem{corollary}[theorem]{Corollary}

\theoremstyle{definition}
\newtheorem{definition}[theorem]{Definition}

\theoremstyle{remark}
\newtheorem{remark}[theorem]{Remark}

\numberwithin{equation}{section}

\def\la{\lambda}
\def\La{\Lambda}
\def\al{\alpha}
\def\si{\sigma}
\def\e{\varepsilon}
\def\de{\delta}
\def\NN{{\mathbb N}}
\def\RR{{\mathbb R}}
\def\CC{{\mathbb C}}
\def\TT{{\mathbb T}}
\def\Re{{\rm Re}\,}
\def\Im{{\rm Im}\,}
\def\Int{{\rm Int}\,}
\def\codim{{\rm codim}\,}
\def\conv{{\rm conv}\,}

\def\dist{{\rm dist}\,}
\def\mT{\mathcal T}

\def\mA{\mathcal A}
\def\mM{\mathcal M}
\def\mK{\mathcal K}
\def\UU{U}
\def\usim{\smash{\mathop{\sim}\limits^u}}
\def\diag{{\rm diag}\,}
\def\DD{{\mathbb D}}
\begin{document}

\title[Joint numerical ranges and compressions of powers]{Joint numerical ranges and compressions of powers of operators}

\author{Vladimir M\"uller}
\address{Institute of Mathematics,\\
Czech Academy of Sciences,\\
 \v Zitna Str. 25, 115 67 Prague,\\
 Czech Republic}
 
\email{muller@math.cas.cz}

\author{Yuri Tomilov}
\address{Institute of Mathematics,\\
Polish Academy of Sciences,\\
\'Sniadeckich Str. 8,
00-956 Warsaw,\\
 Poland}
\email{ytomilov@impan.pl}


\thanks{This work was  partially supported by the NCN grant
 2014/13/B/ST1/03153, by the EU grant  ``AOS'', FP7-PEOPLE-2012-IRSES, No 318910, by grant No. 17-27844S of GA CR and RVO:67985840.}


\begin{abstract}
We identify subsets of the joint numerical range of an operator tuple
in terms of its joint spectrum. This result helps us
to transfer weak convergence
of operator orbits into certain approximation and interpolation properties for powers in the
uniform operator topology. This is a far-reaching  generalization of one of the main results in our recent paper \cite{MullerT}.
Moreover, it yields an essential (but partial) generalization of Bourin's ``pinching'' theorem from \cite{Bourin03}.
It also allows us to revisit
several basic results on joint numerical ranges, provide them with new proofs
and find a number of new results.
\end{abstract}

\maketitle
\section{Introduction}
The theory of joint numerical ranges  is  a developing area of operator theory with several important results obtained in the
last years. The geometric structure of joint numerical ranges has got a considerable attention,
and many properties of numerical ranges have been transferred or appropriately recasted from the setting of a single operator to the framework
of operator  tuples, see e.g. \cite{Binding},  \cite{Li08}-\cite{Li11},  \cite{MullerSt}, \cite{MullerT} and references therein.
At the same time, the relations between  spectrum of an operator tuple and its numerical range remained rather obscure
until very recent time. We are aware of \cite{Wrobel} as the only important contribution to those issues, which moreover was apparently overlooked by the experts.
 Recently in \cite{MullerT}, we have discovered  new spectral inclusion results for operator tuples. Specified for tuples formed
 by powers of a single operator, the results allowed us  to identify the unit circle in the spectrum of a bounded linear
 operator on a Hilbert space in terms of orthogonality and ``quasi-orthogonality'' relations for the operator orbits, see e.g.
  \cite[Theorem 1.1]{MullerT}.
 This constituted an essential generalization of the corresponding results by Arveson \cite{Arveson}, who dealt with unitary operators only.
 Moreover, in \cite{MullerT}, by means of spectral approximations of numerical ranges, we put recent harmonic analysis considerations by Hamdan (\cite{Hamdan})
 into the operator setting and extended them by, in particular, replacing a single orbit in Hamdan's statement by an infinite-dimensional subspace of its orbits. (See below for more on that.)

The present paper brings further insights into relations between spectrum and numerical range for operator tuples,
and uses them to obtain new asymptotic properties of operator orbits under quite general assumptions.
More precisely, we extend, complement and sharpen several main results from \cite{MullerT} on numerical ranges $W(T_1, \dots, T_n)$ of tuples $\mathcal T=(T_1, \dots, T_n)$ of bounded linear operators on a Hilbert space $H$, and in this way obtain essential generalizations of results from \cite{MullerT} on asymptotic behavior of compressions of operator powers. As a consequence, we obtain
a partial generalization of the ``pinching'' theorem by Bourin \cite[Theorem 2.1]{Bourin03} in a much more demanding setting of operator tuples.
For recent applications of \cite[Theorem 2.1]{Bourin03} see \cite{Bourin16}.

One of  the novelties in our approach, stemming from \cite{MullerT}, is that in our studies of geometric properties of operator iterates we rely on the numerical ranges methodology.
It is instructive to note that the condition of orthogonality of  elements from an orbit of $T \in B(H)$ can be rewritten in terms of the joint numerical range of the tuple $\mathcal T=(T, ..., T^n).$
 On the other hand, as we prove below, the joint numerical range $W(\mathcal T)$ contains the interior of the convex hull of the joint spectrum $\sigma(\mathcal T)$ (in spite of the fact that the joint numerical range is in general not convex).
Using inductive arguments, this fact helps us to construct orbits of $T$ with special geometric properties  from the vectors resembling (essential) approximate eigenvectors of $T$.  The constructions are far from being straightforward, and we have to invoke new ideas not present in \cite{MullerT}.
More precisely, our considerations rely on the following spectral inclusion result.

\begin{theorem}\label{main3}
Let $\mathcal T=(T_1,\dots,T_n)\in B(H)^n$. Then
\begin{equation}\label{a1}
\Int  \conv\, \bigl(W_e(\mathcal T)\cup\si_p(\mathcal T)\bigr)\subset W(\mathcal T).
\end{equation}
Moreover, if the tuple $\mathcal T$ is commuting then
\begin{equation}\label{a2}
\Int \conv\, \si(\mathcal T) \subset W(\mathcal T).
\end{equation}
\end{theorem}

Theorem \ref{main3} can be considered as a partial generalization of the main result in \cite[Theorem 2.2]{Wrobel} dealing with
numerical ranges of operators on Banach spaces, and also as a generalization of \cite[Corollaries 4.2 and 4.4]{MullerT}, where $\si_p(\mathcal T)$ was absent in \eqref{a1}, and \eqref{a2} was stated with $\si(\mathcal T)$ replaced by $\si_e(\mathcal T).$ Note that while the result in \cite{Wrobel} allows to
find parts of the spectrum of $\mathcal T$ in $\overline{W(\mathcal T)},$ we may replace $\overline{W(\mathcal T)}$ by a smaller and more transparent set $W(\mathcal T).$ The proof of Theorem \ref{main3} requires new tools, e.g. Zenger's Lemma, and it is technically more demanding than the corresponding arguments in \cite{MullerT}.

To present our applications of Theorem \ref{main3} (or rather its predecessor from \cite{MullerT}), let us recall that,
motivated by applications in ergodic theory, Hamdan characterized in \cite{Hamdan} the size of the spectrum of some unitary operators by a new type of asymptotic assumptions. He proved  that if a unitary operator $T$ on $H$ is such that $T^n \to 0$ in the weak operator topology,
then  $\sigma(T)=\mathbb T$  if and only if for every $\epsilon >0$ there exists a unit vector $x \in H$
satisfying
\begin{equation}\label{hamda}
\sup_{n \ge 1} |\langle T^n x,x \rangle| <\epsilon.
\end{equation}
His result has been extended in \cite{MullerT} to general bounded operators and to the setting allowing to take $x$'s in \eqref{hamda} from an infinite-dimensional subspace. Namely, we proved  in \cite[Corollary 6.3 and Remark 6.4]{MullerT} that
if $T$ is a bounded linear operator on $H$  such that $T^n \to 0$ in the weak operator topology, and
$\sigma(T) \supset \mathbb T,$ then for every $\epsilon >0$ one can find an infinite-dimensional subspace $L$ of $H$ such that
the compressions $(T^n)_L$ of $T^n$ to $L$ are asymptotically small in two senses:
$$
\lim_{n \to \infty} \|(T^n)_L\|=0 \qquad \text{and} \qquad \sup_{n \ge 1} \|(T^n)_L\|<\epsilon.
$$
This result can be generalized as follows.
\begin{theorem}\label{hamdanintro}
Let $T\in B(H)$ be such that $T^n\to 0$ in the weak operator topology, and let  $\si(T)\supset\TT$. Let  $\tilde C$ be a strict contraction on a separable Hilbert space, i.e., $\|\tilde C\|<1$. Then for every $\e>0$ there exists a subspace $L\subset H$ and $C\in B(L)$ unitarily equivalent to $\tilde C$ such that
\begin{equation}\label{asymp}
\lim_{n\to\infty} \|(T^n)_L- C^n\|=0\qquad\hbox{and}\qquad
\sup_{n \ge 1} \|(T^n)_L- C^n\|\le \e.
\end{equation}
\end{theorem}

It is natural to ask whether the asymptotic relation \eqref{asymp} above can be made an exact equality. Surprisingly, the answer is ``yes'', if
one restricts oneself to a finite piece of the orbit $(C^n)_{n \ge 1}$. In particular, the following theorem holds.
\begin{theorem}\label{bourinintro}
Let $T\in B(H)$ and suppose that the polynomial hull $\hat\sigma(T)$ of $\sigma(T)$ contains the unit disc $\DD$. Let $n \in\NN$ and let $\tilde C$ be a strict contraction on a separable Hilbert space. Then there exists a subspace $L\subset H$ and $C\in B(L)$ unitarily equivalent to $\tilde C$ such that $(T^k)_L=C^k$ for $k=1,\dots,n$.
\end{theorem}
The question when it is possible to obtain the equality $(T^k)_L=C^k$
for all $k\in\NN$ (i.e., when $T$ is a dilation of $C$) was studied by
using the Scott Brown technique. In particular, in \cite[Theorem 4.8]{Bercovici} a positive result
was obtained for so called BCP-operators (contractions with dominant
essential spectrum).

Despite the main motivation for the paper was to understand how far Hamdan's type results can be pushed by operator-theoretical technique,
as a byproduct of our approach we found new arguments for the proofs of recent characterizations of essential and infinite numerical ranges,
 as well as several new statements concerning numerical ranges which are of independent interest.
 Recall that the infinite numerical range $W_\infty(\mathcal T)$ of  a tuple $\mathcal T=(T_1, \dots, T_n)$ can be defined as
 $$
W_\infty(\mathcal T):=\bigl\{(\lambda_1, \dots, \lambda_n)\in \mathbb C^n: P T_jP =\la_jP, \quad j=1,\dots,n\bigr\}
$$
 for some infinite rank projection $P.$ We prove that the essential numerical range $W_e(\mT)$ of $\mathcal T$ can be described
 in terms of $W_\infty (\mathcal T)$ as
 $$
 W_e(\mT)=\bigcup_{\mK\in\mK(H)^n} W_\infty(\mT-\mK).
 $$
 Moreover there exists an $n$-tuple $\mK$ of trace-class operators on $H$ such that
$W_e(\mT)=\overline{W_\infty(\mT-\mK)}.$
 We also show that for every tuple $\mathcal T \in B(H)^n$ of bounded linear operators on $H$ one has
$$
\overline{\rm conv}\, W(\mathcal T) = {\rm conv}\,\bigl(W(\mathcal T)\cup W_e(\mathcal T)\bigr),
$$
and
$$
\Int \overline {W(\mathcal T)}\subset W(\mathcal T)
$$
if $\Int W_e(T)\ne\emptyset$.

\section{Notation}\label{nota}
It will be convenient to fix some notations in a separate section. In particular, we let
$H$ be a Hilbert space with the inner product $\langle\cdot ,\cdot \rangle,$ and $B(H)$ the space of all bounded linear operators on $H$.
For a bounded linear operator $T$ we denote by $\sigma (T)$ its spectrum, and by $N(T)$ its kernel.


In the following we consider an $n$-tuple $\mathcal T=(T_1,\dots,T_n)\in B(H)^n$. Note that we do not in general
assume that the operators $T_j$ commute. For $x,y\in H$ we write shortly $\langle \mathcal Tx,y\rangle= (\langle T_1x,y\rangle,\dots,\langle T_nx,y\rangle)\in\CC^n$ and ${\mathcal T}x=(T_1x,\dots,T_nx)\in H^n$.
Similarly for $\la=(\la_1,\dots,\la_n)\in\CC^n$
we write $\mathcal T-\la=(T_1-\la_1,\dots,T-\la_n)$ and $\|\la\|=\max\{|\la_1|,\dots,|\la_n|\}$.
If  $\mT=(T_1,\dots, T_n)\in B(H)^n$ and $R, S \in B(H)$ then
\begin{equation}\label{multip}
 R \mT  S :=(R T_1 S, \dots, R T_n S).
\end{equation}
Thus, in particular, if $\mT \in B(H)^n$  and $P_M$ is the orthogonal projection from $H$ onto $M$ then
$P_M \mT P_M=(P_M T_1 P_M, \dots, P_M T_n P_M).$

For a closed set $K \subset \mathbb C^n $ we denote by $\partial K$ the topological boundary of $K$,  by ${\rm Int} \, K$ the interior of $K,$
by ${\rm conv}\, K$ the convex hull of $K,$ and by $\widehat K$
the polynomial hull of $K$. Recall that if $K\subset\CC$ then $\widehat K$ is the union of $K$ with all bounded components of the complement $\CC\setminus K$.

Finally, we let $\mathbb T$ stand for the unit circle $\{\lambda \in \mathbb C: |\lambda|=1\}$, $\mathbb D$ for the unit disc $\{\lambda\in\CC: |\lambda|<1\}$ and $\mathbb R_+=[0,\infty).$

\section{Preliminaries}
We start with recalling certain basic notions and facts from the spectral theory of operator tuples on Hilbert spaces.
They can be found e.g. in \cite[Chapters 2-3]{Muller}.

Let $\mathcal T=(T_1,\dots,T_n)\in B(H)^n$ be an $n$-tuple of commuting operators.
Recall that  its  joint (Harte) spectrum $\sigma(\mathcal T)$ can be defined as the complement of the set of those $\lambda=(\la_1,\dots,\la_n) \in\mathbb C^n$ for which
$$
\sum_{j=1}^n L_j(T_j-\lambda_j)=\sum_{j=1}^{n}(T_j-\lambda_j)R_j=I
$$
for some $L_j, R_j, 1 \le j \le n,$ from the algebra $B(H).$
There are two  particularly useful subsets of $\sigma(\mathcal T).$
 The first one, the joint essential spectrum $\si_e(\mathcal T)$ of $\mathcal T,$ can be defined  as the (Harte) spectrum of the $n$-tuple $\pi(\mathcal T):=(\pi(T_1),\dots,\pi (T_n))$ in the Calkin algebra $B(H)/{\mathcal K}(H)$, where ${\mathcal K}(H)$ denotes the ideal of all compact operators on $H$, and $\pi: B(H) \to B(H)/{\mathcal K}(H)$
stands for the quotient map.
The second one, the essential approximate spectrum $\si_{\pi e}(\mathcal T)$ of $\mathcal T$ consists of all $\la=(\la_1,\dots,\la_n)\in\CC^n$ such that there exists an orthonormal sequence $(x_k)_{k \ge 1}\subset H$ satisfying
$$\sum_{j=1}^n\|(T_j-\la_j)x\|=0.$$
 It is easy to show that $\si_{\pi e}(\mathcal T)\subset \si_{e}(\mathcal T).$
Note that if $n=1$ then
$\si_{e}(T_1)=\{\la_1\in\CC: T_1-\la_1\hbox{ is not Fredholm}\},$
and
for $T\in B(H)$ and $\mathcal T=(T,T^2, \dots, T^n) \in B(H)^n,$ one has $\sigma (\mathcal T)=\{(\lambda, \dots, \lambda^n): \lambda \in \sigma (T)\},$  where $\sigma$ can be replaced by either
$\sigma_e$ or $\sigma_{\pi e}.$
It is well-known that $\sigma({\mathcal T})$ and  $\sigma_e({\mathcal T})$ are non-empty compact subsets of $\mathbb C^n,$
while  $\sigma_{\pi e}({\mathcal T})$ can be empty even if $n=1.$ Basic facts on essential spectra of operator tuples can be found
in \cite[Chapter III.19]{Muller}.

For  not necessarily commuting $n$-tuple $\mathcal T$
denote by $\si_p(\mathcal T)$ the point spectrum of $\mathcal T$, i.e., the set of all $n$-tuples $\lambda=(\la_1,\dots,\la_n)\in\CC^n$ such that $\bigcap_{j=1}^n N(T_j-\la_j)\ne\{0\}$. If $ x \in \bigcap_{j=1}^n N(T_j-\la_j)$ then
we will write $\mT x=\lambda x.$ Remark, however, that in fact we will not need a somewhat cumbersome spectral theory of non-commuting operator tuples.

As in the case of a single operator, it is often useful to relate $\sigma(\mathcal T)$ to a larger and more easily computable set $W(\mathcal T)\subset \mathbb C^n$ called the  joint numerical
range of $\mathcal T$ and defined as
$$
W(\mathcal T)=\bigl\{(\langle T_1 x, x\rangle , ..., \langle T_n x, x \rangle) : x \in H, \|x\|=1\bigr\}.
$$
The set $W(\mathcal T)$ can be identified with a subset of $\mathbb R^{2n}$ if one identifies the $n$-tuple $\mathcal T$ with the $2n$-tuple
$({\rm Re}\, T_1, {\rm Im}\, T_1, ..., {\rm Re}\, T_n, {\rm Im}\, T_n)$ of selfadjoint operators.
Unfortunately, if  $n>1,$ then $W(\mathcal T)$ is not in general convex, see e.g. \cite{Li09}.

As in the spectral theory, there is also a notion of the joint essential numerical range  $W_e(\mathcal T)$  associated to $\mathcal T.$
For $\mathcal T=(T_1,\dots,T_n)\in B(H)^n$ we define  $W_e(\mathcal T)$ as the set of all $n$-tuples
$\la=(\la_1,\dots,\la_n)\in\CC^n$ such that there exists an orthonormal sequence $(x_k)_{k \ge 1}\subset H$ with
$$
\lim_{k\to\infty}\langle  T_j x_k,  x_k\rangle=\la_j, \qquad j=1,\dots,n.
$$
Alternatively, $W_e(\mathcal T)$ can be defined as
$$
W_e(\mathcal T):= \bigcap \overline{W(T_1 + K_1,\dots,T_n+K_n)}
$$
where the intersection is taken over all $n$-tuples $K_1,\dots,K_n$ of compact operators on $H.$
Recall that $W_e(\mathcal T)$ is a nonempty, compact and, in contrast to $W(\mathcal T),$ \emph{convex} subset of $\overline{W(\mathcal T)}$, see \cite{Bercov} or \cite{Li09}.
Note that as a straightforward consequence of the definitions above, if the $n$-tuple $\mathcal T \in B(H)^n$ is commuting then  $\si_{\pi e}(\mathcal T)\subset W_e(\mathcal T).$ Then the convexity of $W_e(\mathcal T)$ implies that $\conv \si_e(T) \subset W_e(T),$ see the proof of Corollary \ref{spectrum} below.

There's also a useful and related notion of the numerical range for tuples of elements of a unital Banach algebra $\mathcal A.$
For $a=(a_1, ..., a_n)\in \mathcal A^n$ define
\begin{equation}\label{brange}
V(a, \mathcal A):=\bigl\{(f(a_1), \dots, f(a_n)): f\in \mathcal A^*, f(1)=\|f\|=1\bigr\},
\end{equation}
and recall that $f \in A^*$ such that $f(1)=\|f\|=1$ are called states.
With such a definition, one has
\begin{equation}\label{calkin}
\overline{\rm conv}\, W(\mathcal T)=V(\mathcal T, B(H))
\end{equation}
and
\begin{equation}\label{calkin1}
W_e(\mathcal T)=V\bigl(\pi(\mathcal T), B(H)/\mathcal K(H)\bigr).
\end{equation}
The proofs of \eqref{calkin} and \eqref{calkin1} can be found  in \cite[Theorem 1 and Theorem 2]{MullerSt}, respectively.
For a comprehensive account of joint essential numerical ranges one may consult \cite{Li09}.
Very recently,  several geometric properties of joint essential numerical ranges
(as e.g. convexity) were extended in \cite{Lau} to the setting of joint matricial essential
ranges.

The next result due to Zenger is used in a number of operator-theoretical constructions. Its proof can be found  e.g. in \cite[p. 18-20]{Bonsall}.
\begin{lemma}[(Zenger's Lemma)]
Let $u_1, \dots , u_n\in H$ be linearly independent, and let $\alpha_1, \dots , \alpha_n \in \mathbb R_+$ be such that
$\sum_{k=1}^n \alpha_k=1.$ Then there exist $ w_1, \dots , w_n \in \mathbb C$ and $u \in H, \|u\| \le 1,$
satisfying $\|\sum_{j=1}^n w_ju_j\|\le 1$ and
$$
\langle w_j u_j, u \rangle=\alpha_j, \qquad 1 \le j \le n.
$$
\end{lemma}
Note that $\bigl\langle\sum_{j=1}^nw_ju_j,u\bigr\rangle=1$, and so $u=\sum_{j=1}^nw_ju_j$.

\section{Spectra and numerical ranges for tuples}\label{numersection}

In the following we consider an $n$-tuple $\mathcal T=(T_1,\dots,T_n)\in B(H)^n$. Note that we do not assume that the operators $T_j$ commute.

It was proved in \cite[Corollary 4.2]{MullerT} that
\begin{equation}\label{we}
\Int \, W_e(\mathcal T) \subset W(\mathcal T).
\end{equation}
Moreover,  if $\la=(\la_1, \dots, \la_n)\in\Int \, W_e(\mathcal T),$
then there exists an infinite-dimensional subspace $L$ of $H$ such that
\begin{equation}\label{ep}
P_LT_jP_L=\la_jP_L, \qquad j=1,\dots,n,
\end{equation}
where $P_L$ is the orthogonal projection on $L.$
So, despite $T_j, 1 \le j \le n,$ do not commute, in a number of situations of interest they have
diagonal compressions to the same subspace. Moreover, \eqref{we} has important spectral consequences. It was shown in \cite[Corollary 4.4]{MullerT}
that
\begin{equation}\label{sp}
\Int\,\conv\, \sigma_e(\mathcal T)\subset W(\mathcal T).
\end{equation}

On the other hand, \eqref{we} has certain drawbacks. For instance, if one of the operators $T_j, 1 \le j \le n,$ is compact then
$\Int W_e(T_1, \dots, T_n)=\emptyset,$ and \eqref{we} says nothing. Thus, it is desirable, to
obtain extensions of \eqref{we} shedding also light on $W(\mathcal T)$ in the case of tuples with ``small'' essential numerical range.
The next theorem serves just that purpose.  Extending \eqref{we}, it allows one to describe "big" subsets of  $W(\mathcal T)$ in spectral terms.
The result is also related to \cite[Theorem 2.2]{Wrobel} where weaker statements have been obtained.
As \cite[Theorem 2.2]{Wrobel}, the theorem below depends on Zenger's Lemma, and also uses the following simple statement.
\begin{lemma}\label{compres}
Let $\mT \in B(H)^n.$ If $\lambda \in W_e(\mT),$ then
for every $\de>0$ and every  subspace $M\subset H$ of finite codimension
there exists a unit vector $x\in M$ such that $\|\langle\mT x,x\rangle-\la\|<\de$.
\end{lemma}
\begin{proof}
 If $\la\in W_e(\mT)$ then there exists an orthonormal sequence $(x_i)_{i \ge 1}$ in $H$  such that $\langle\mT x_i,x_i\rangle\to \la$, $i \to \infty$. Let $M\subset H$ be a subspace of a finite codimension and $\de>0$.
 We have $\|P_{M^\perp}x_i\|\to 0$, and so $\|P_Mx_i-x_i\|\to 0$ as $i \to \infty$.
 Set $$u_i=\frac{P_Mx_i}{\|P_Mx_i\|}, \qquad i \ge 1.$$
Then $\lim_{i\to\infty}\|u_i-x_i\|=0$ and $\lim_{i\to\infty}\langle \mT u_i,u_i\rangle=\la$. Hence there exists $i_0$ such that
$\|\langle\mT u_{i_0},u_{i_0}\rangle-\la\|<\de$.
\end{proof}
\begin{theorem}\label{numranges}
Let $\mathcal T=(T_1,\dots,T_n)\in B(H)^n$. Then
$$\Int \conv\bigl(W_e(\mathcal T)\cup\si_p(\mathcal T)\bigr)\subset W(\mathcal T).$$
\end{theorem}

\begin{proof}
Let $\la=(\la_1,\dots,\la_n)\in\Int\conv\bigl(W_e(\mathcal T)\cup\si_p(\mathcal T)\bigr)$. We show that $\la\in W(\mathcal T)$.
By considering the $n$-tuple $(T_1-\la_1,\dots,T_n-\la_n)$ instead of $(T_1,\dots,T_n)$, we can assume without loss of generality that $(\la_1,\dots,\la_n)=(0,\dots,0).$

Let $r>0$ satisfy
$$
\{(\e_1,\dots,\e_n):\max_j|\e_j|\le r\}\subset\conv \bigl(W_e(T)\cup\si_p(T)\bigr).
$$

Let $\la^{(1)},\dots,\la^{(m)}\in\si_p(\mathcal T)\setminus W_e(\mathcal T)$ be a finite set such that
$$
\min\bigl\{\|\mu-\la^{(i)}\|:i=1,\dots,m\bigr\}<\frac{r}{2}
$$
for all $\mu\in\si_p(\mathcal T)$.

We show that
\begin{equation}\label{inclu}
\Bigl \{\e\in\CC^n: \|\e\|\le\frac{r}{2}\Bigr\}\subset
\conv\bigl(W_e(\mathcal T)\cup\{\la^{(1)},\dots,\la^{(m)}\}\bigr).
\end{equation}
Let $a_1,\dots,a_n\in\CC$, $\sum_{i=1}^n|a_i|=1$, and $c\in\RR.$ Assume that
$$
\Re \sum_{i=1}^na_iz_i \ge c
$$
for all $z=(z_1,\dots,z_n)\in W_e(\mathcal T)\cup\{\la^{(1)},\dots,\la^{(m)}\}$.
Then
$$
\Re \sum_{i=1}^na_iz_i \ge c-\frac{r}{2}
$$
for all $z=(z_1,\dots,z_n)\in W_e(\mathcal T)\cup\si_p(\mathcal T)$. Consequently,
$$\Re \sum_{i=1}^na_iz_i \ge c-\frac{r}{2} \qquad \text{for all}\qquad z\in\CC^n, \quad \|z\|\le r.$$
Setting $z_i=e^{-i\arg a_i} r, 1 \le i \le n,$ in the inequality above we infer that
$c+\frac{r}{2}\le 0$.

If $z\in\CC^n$, $\|z\|\le \frac{r}{2}$, then
$$\Re \sum_{i=1}^na_iz_i \ge -\frac{r}{2}\ge c.$$

Since the convex  hull of $W_e(\mathcal T)\cup\{\la^{(1)},\dots,\la^{(m)}\}$ is the intersection of all halfspaces containing it, this shows \eqref{inclu}.

Fix now eigenvectors $u_1,\dots,u_m\in H$ such that $\mathcal T  u_i=\la^{(i)} u_i,$ $\|u_i\|=1$.
Let
$$F=\bigvee_{i=1}^m u_i \qquad \text{and} \qquad  M=F^\perp\cap\bigcap_{j=1}^n (T_jF)^\perp\cap\bigcap_{j=1}^n (T_j^{*}F)^\perp.$$
Clearly $\dim F<\infty$ and $\codim M<\infty$.
Note that
$$F\perp M, \quad T_j F \perp M \quad \text{and} \quad T_jM\perp F \quad  \text{for all}\, j=1,\dots,n.$$
We construct a unit vector $x\in H$ satisfying $\langle T_jx,x\rangle=0$ for all $1 \le j \le n$ inductively as a limit of consecutive approximations. Set $x_0=v_0=w_0=0$. It will be convenient to separate the following fact.
\bigskip

\noindent{\bf Claim.} Let $r>0$, $F, M\subset H$ be as above. Let
$k\in \mathbb N\cup\{0\}$, let  $v_{k}\in M$, $w_{k}\in F$ and $x_{k}=v_k+w_k$
satisfy
$$
w_k=\sum_{i=1}^m t_iu_i, \qquad \langle t_iu_i,w_k\rangle=\al_i\ge 0, \qquad  \|w_k\|^2=\sum_{i=1}^m\al_i,
$$
$$
\|x_k\|^2=1-2^{-k} \qquad \text{and} \qquad \|\langle \mathcal  T x_k, x_k\rangle\|<\frac{r}{2^{k+2}}.
$$

Then there exist $v_{k+1}, w_{k+1}$ and $x_{k+1},$
$$x_{k+1}=v_{k+1}+w_{k+1}, \qquad v_{k+1}\in M, \,\, w_{k+1} \in F,$$
such that
$$
w_{k+1}=\sum_{i=1}^ms_iu_i, \qquad \langle s_iu_i,w_{k+1}\rangle=\beta_i\ge 0, \qquad \|w_{k+1}\|^2=\sum_{i=1}^m\beta_i,
$$
$$
\|x_{k+1}\|^2=1-2^{-k-1}, \qquad  \|v_{k+1}-v_k\|^2\le\frac{1}{2^{k+1}}
$$
and
$$
\|\langle \mathcal T x_{k+1}, x_{k+1}\rangle\|<\frac{r}{2^{k+3}}.$$
\medskip

To prove the claim, let $\e=\langle \mathcal Tx_k,x_k\rangle$. Since $\|\e\|<\frac{r}{2^{k+2}},$ there exist
elements $\la^{(m+1)},\dots,\la^{(m')}\in W_e(\mathcal T)$ and numbers $c_1,\dots,c_{m'}\ge 0$ such that $\sum_{i=1}^{m'}c_i=1$ and
$$
\sum_{i=1}^{m'}c_i\la^{(i)}=-\e 2^{k+1}.
$$
By  Zenger's Lemma, there are complex numbers $s_1,\dots,s_n$ such that $w_{k+1}:=\sum_{i=1}^ms_iu_i$ satisfies
$$\langle s_iu_i,w_{k+1}\rangle=\al_i+\frac{c_i}{2^{k+1}} \quad \text{and}\quad
\|w_{k+1}\|^2=\sum_{i=1}^m\Bigl(\al_i+\frac{c_i}{2^{k+1}}\Bigr).$$
Thus $$\langle \mathcal Tw_{k+1},w_{k+1}\rangle=\sum_{i=1}^m\Bigl(\al_i+\frac{c_i}{2^{k+1}}\Bigr)\la^{(i)}.$$

The elements $\la^{(m+1)},\dots,\la^{(m')}$ belong to $W_e(\mathcal T)$.
Using Lemma \ref{compres}  and the induction argument, we can construct unit vectors $y_{m+1},\dots,y_{m'}$ in the following way.

Suppose that $m\le s<m'$ and that the vectors $y_{m+1},\dots,y_s$ have already been constructed.
Set
\begin{align*}
G:=&\bigvee\{v_k, T_jv_k, y_i, T_jy_i:m+1\le i\le s, 1\le j\le n\}, \\
L:=&G^\perp\cap\bigcap_{j=1}^n(T_j^*G)^\perp.
\end{align*}

 Then $\dim G<\infty$ and $\codim L<\infty$. Hence there exists a unit vector $y_{s+1}\in L\cap M$ such that
$$
\bigl\|\langle T  y_{s+1},  y_{s+1}\rangle-\la^{(s+1)}\bigr\|<\frac{r}{4}.
$$
If the vectors $y_{m+1},\dots,y_{m'}$ are constructed, set
$$v_{k+1}=v_k+\sum_{i=m+1}^{m'}\Bigl(\frac{c_i}{2^{k+1}}\Bigr)^{1/2}y_i \qquad \text{and} \qquad x_{k+1}=v_{k+1}+w_{k+1}.$$
 Then
 \begin{align*}
\|v_{k+1}\|^2=&\|v_k\|^2+\sum_{i=m+1}^{m'}\frac{c_i}{2^{k+1}},\\
\|v_{k+1}-v_k\|^2=&\sum_{i=m+1}^{m'}\frac{c_i}{2^{k+1}}\le\frac{1}{2^{k+1}}
\end{align*}
and
\begin{align*}
\|x_{k+1}\|^2=&\|v_{k+1}\|^2+\|w_{k+1}\|^2=
\|v_k\|^2+\sum_{i=m+1}^{m'}\frac{c_i}{2^{k+1}}+\sum_{i=1}^{m}\Bigl(\al_i+\frac{c_i}{2^{k+1}}\Bigr)\\
=&
\|v_k\|^2+\frac{1}{2^{k+1}}+\|w_k\|^2=
\|x_k\|^2+\frac{1}{2^{k+1}}=
1-\frac{1}{2^{k+1}}.
\end{align*}

Finally,
\begin{align*}
\bigl\|\langle {\mathcal T}  x_{k+1}, x_{k+1}\rangle\bigr\|=&
\Bigl\|\langle {\mathcal T} v_{k}, v_{k}\rangle+ \sum_{i=m+1}^{m'}\frac{c_i}{2^{k+1}}\langle {\mathcal T} y_i, y_i\rangle+\langle {\mathcal T} w_{k+1}, w_{k+1}\rangle\Bigr\|\\
\le&
\Bigl\| \langle {\mathcal T} v_{k}, v_{k}\rangle+ \sum_{i=m+1}^{m'}\frac{c_i\la^{(i)}}{2^{k+1}}
+\sum_{i=1}^m(\al_i+\frac{c_i}{2^{k+1}})\la^{(i)}\Bigr\|\\
+&
\sum_{i=m+1}^{m'} \frac{c_i}{2^{k+1}}\bigl\|\langle T  y_i, y_i\rangle-\la^{(i)}\bigr\| \\
\le&
\Bigl\|\langle T x_k,  x_k\rangle+\frac{1}{2^{k+1}}\sum_{i=1}^{m'}c_i\la^{(i)}\Bigr\|+\sum_{m+1}^{m'}\frac{c_i}{2^{k+1}}\cdot \frac{r}{4}\\
=&\frac{r}{2^{k+3}}.
\end{align*}
This finishes the proof of the claim.

\medskip

Now construct the vectors $v_k,w_k$ and $x_k=v_k+w_k, k \in \mathbb N,$ inductively as described in the Claim.
Clearly $(v_k)_{k \ge 1}$ is a Cauchy sequence, and we let $v\in M$ be its limit.
The sequence $(w_k)_{k \ge 1}$ is a bounded sequence in the finite-dimensional space $F$. By passing to a subsequence if necessary, we may assume that $(w_k)_{k \ge 1}$ is convergent, $w_k\to w\in F, k \to \infty$.
The vector
$$  x= v+  w=\lim_{k \to \infty}(v_k + w_k)=\lim_{k \to \infty} x_k$$
satisfies $\|x\|=1$ and $\langle \mathcal T  x, x\rangle=0$.
This finishes the proof.
\end{proof}

Theorem \ref{numranges} yields the following generalization of \cite[Corollary 4.4]{MullerT},
 replacing $\sigma_e(\mathcal T)$ by $\sigma(\mathcal T)$ there, cf. \eqref{sp}.
 The generalization  complements \cite[Corollary 2.3]{Wrobel} where, for a commuting tuple $\mathcal T \in B(H)^n,$ it was shown that
 \begin{equation}\label{wrobel}
 \conv\si(\mathcal T)\subset \overline {W(\mathcal T)}.
 \end{equation}
(Note that, as it will be clear below, we will be interested in spectral inclusions for the numerical range $W(\mathcal T),$ rather than for its closure.)

\begin{corollary}\label{spectrum}
Let $\mathcal T=(T_1,\dots,T_n)\in B(H)^n$ be a commuting $n$-tuple. Then $$\Int\conv\si(\mathcal T)\subset W(\mathcal T).$$
\end{corollary}

\begin{proof}
First note that by \cite[Corollary 19.16]{Muller}, the polynomial hulls $\widehat \si_e(\mathcal T)$ and $\widehat \si_{\pi e}(\mathcal T)$ coincide,
so $\conv \si_e(\mathcal T)=\conv \si_{\pi e}(\mathcal T).$ In view of convexity of $W_e(\mathcal T)$, it follows that
$\conv \si_e(\mathcal T)\subset W_e(\mathcal T),$ and thus, in particular, $\widehat \si_{e}(\mathcal T) \subset W_e(\mathcal T)$.
Moreover, by \cite[Theorem 19.18]{Muller}, the set $\si (\mathcal T)\setminus \widehat \sigma_e(\mathcal T)$ consists of isolated eigenvalues of $\mathcal T$. Therefore, we have
$$\conv \si(\mathcal T)= \conv (\widehat \si_{e}(\mathcal T) \cup\si_p(\mathcal T)) \subset \conv (W_{e}(\mathcal T) \cup\si_p(\mathcal T)).$$
The statement follows then from Theorem \ref{numranges}.
\end{proof}

Statements like Theorem \ref{numranges} specified for tuples $(T,T^2,\dots,T^n)\in B(H)^n$
allow one to find appropriate tuples of powers of complex numbers in their joint numerical ranges $W(T,T^2,\dots,T^n)$,  thus revealing
certain geometric properties of the orbits of $T.$ For instance, the fact that $(0, \dots, 0) \in W(T,T^2,\dots,T^n)$
yields an element $x \in H$ such that $x \perp T^k x$ for all $k$ between $1$ and $n.$
The latter property was introduced and characterized in spectral terms for unitary $T$ by Arveson,  \cite{Arveson}.
For its  generalizations  see \cite{MullerT}. In general, the structure of $W(T,T^2,\dots,T^n)$
can be rather complicated even if $H$ is finite-dimensional, see e.g. \cite{Davis}.

The next theorem was proved in \cite[Corollary 4.2 and Corollary 4.7]{MullerT}. It will be instrumental in Section \ref{restrict}
below dealing with asymptotic properties of compressions of powers.

\begin{theorem}\label{lambdawe}
Let $T\in B(H)$ and let $\la$ belong to the interior of polynomial hull of $\sigma(T)$.  Then
$$
(\la,\la^2,\dots,\la^n)\in  \Int W_e(T, T^2, \dots, T^n) \subset W(T,T^2,\dots,T^n).
$$
for all $n\in\NN.$
\end{theorem}
Note  that the assumption on $\la$ as in the theorem above is quite natural and apparently  close to optimal as the following statements show.

\begin{proposition}
Let $T\in B(H)$, $n\in\NN$. Suppose that
\begin{enumerate}
\item [(i)] $(0,\dots,0)\notin{\rm Int}\,{\rm conv}\, W(T, T^2, \dots,T^n)$;
\item [(ii)] there exists a unit vector $x\in H$ such that $x\perp Tx,T^2x,\dots,T^{2n}x$.
\end{enumerate}
Then $\sigma_p(T)\ne\emptyset$.
\end{proposition}

\begin{proof}
Suppose $x\in H$ satisfies (ii). Then, by (i), there exist complex numbers $c_1,\dots,c_n$ such that
$$
{\rm Re}\,\sum_{j=1}^n \langle c_jT^ju,u\rangle\ge 0
$$
for all $u\in H$.
Let $S=\sum_{j=1}^nc_jT^j$. Let $\alpha\in\CC$ and $y=\alpha x+Sx$.
Then
$$
0\le
{\rm Re}\, \langle Sy,y\rangle =
{\rm Re}\,\langle \alpha Sx+S^2x,Sx\rangle=
{\rm Re}\,\Bigl(\alpha\|Sx\|^2+\langle S^2x,Sx\rangle\Bigr).
$$
Since this is true for all $\alpha\in \CC$, we have $Sx=0$. So $0\in\sigma_p(S)=\sigma_p(\sum_{j=1}^nc_jT^j)=\bigl\{\sum_{j=1}^n c_j\lambda^j:\lambda\in\sigma_p(T)\bigr\}$ by the spectral mapping theorem for the point spectrum. Hence $\sigma_p(T)\ne\emptyset$.
\end{proof}

\begin{corollary}\label{cor}
Let $T\in B(H)$, $\sigma_p(T)=\emptyset$. Suppose that for all $n \in \mathbb N$ one has $(0,\dots, 0) \in W(T,\dots,T^n)$. Then
$$
(0,\dots,0)\in{\rm Int}\,{\rm conv}\, W(T,\dots,T^n)
$$
for all $n\in\NN$.
\end{corollary}

\section{Joint numerical ranges revisited}
In this section we use the results proved above to provide alternative and, we believe, sometimes simpler proofs
of the theorems describing essential and so-called infinite numerical ranges for tuples in terms of their compressions and higher rank numerical ranges.
The results were (essentially) obtained in
\cite{Li09} and \cite{Li11}, see also
\cite{Li08},  \cite{Roldan}, and \cite{Wo08} for their single operator analogues,
and \cite{MullerSt} for complementary results.
Moreover, our techniques allow us
to prove several new results of independent interest.

Let for the rest of this section $H$ will stand for an infinite-dimensional separable Hilbert space. We start with general considerations on joint numerical ranges.
Recall that the joint numerical range $W(\mathcal T)$ is, in general, neither convex nor closed.
Thus, it makes sense to describe the closed convex hull of $W(\mathcal T)$ in terms of $W(\mathcal T)$
and the related set $W_e(\mathcal T).$
The following statement is an extension of a similar theorem due to Lancaster for single operators \cite{Lancaster}.
Its proof is based on an idea of Williams from \cite{Williams}.
For a different, geometrical proof of the statement see  \cite{Cho} (and also \cite[Theorem 2.1 and Corollary 2.3]{Chan} for related results).
\begin{theorem}
Let $\mathcal T =(T_1,\dots,T_n)\in B(H)^n$. Then
$$
\overline{\rm conv}\, W(\mathcal T) = {\rm conv}\,\bigl(W(\mathcal T)\cup W_e(\mathcal T)\bigr).
$$
\end{theorem}
\begin{proof}
Since $W_e(\mathcal T)\subset \overline{W(\mathcal T)}$, we have the inclusion
``$\supset$''.

Conversely, let $(\lambda_1,\dots,\lambda_n)\in \overline{\rm conv}\, W(\mathcal T).$ Recall that by \eqref{calkin} one has
$\overline{\rm conv}\, W(\mathcal T)=V(\mathcal T, B(H))$. So there exists a state $f\in B(H)^*$ such that $f(T_j)=\lambda_j$ for all $j=1,\dots,n$. By Dixmier's theorem \cite{Dixmier},  one has a decomposition $f=\alpha f_0+(1-\alpha)f_1$, where $0\le\alpha \le 1$ and $f_0, f_1$ are states on $B(H)$ such that $f_0$ annihilates the ideal of compact operators $\mathcal K(H)$ on $H$, and  $f_1(A):={\rm trace}\, (AS)$ for a fixed trace class operator $S\ge 0$ and all $A\in B(H).$  Hence there exist an orthonormal system $(e_k)_{k \ge 1}\subset H$ and positive numbers $\beta_k$ with $\sum_{k\ge 1}\beta_k=1$ such that $S=\sum_{k\ge 1} \beta_k e_k\otimes e_k$.
Thus
$$f_1(A)=\sum_{k \ge 1}\beta_k\langle A e_k,e_k\rangle$$
for all $A\in B(H)$.
Recall that  a convex set in $\mathbb C^n$  is invariant
with respect to taking infinite convex combinations of its elements (note that the set may be not closed),
see e.g. \cite{Cook} or \cite{Rubin}. Thus, since clearly $\{(T e_k,e_k): k \ge 1\} \subset  {\rm conv}\, W(\mathcal T),$ we have
$$f_1(\mathcal T) \in {\rm conv}\, W(\mathcal T).$$

By \eqref{calkin1},  $V\bigl(\pi(T_1),\dots, \pi(T_n), B(H)/\mathcal K(H)\bigr)=W_e(\mathcal T)$,
where $\pi: B(H) \to B(H)/\mathcal K(H)$ is the quotient map.
Hence
$f_0(\mathcal T) \in W_e(\mathcal T),$
and thus
$$
(\lambda_1,\dots,\lambda_n)= (f(T_1),\dots, f(T_n))\in
{\rm conv}\,\bigl(W(\mathcal T)\cup W_e(\mathcal T)\bigr).
$$
\end{proof}

Despite the properties of joint numerical ranges are much more involved than the properties
of numerical ranges for single operators, joint numerical ranges can be described in terms of other numerical ranges
that are somewhat simpler to deal with.
Let us recall now the definition of higher rank numerical ranges.

\begin{definition}
Let $\mathcal T=(T_1,\dots,T_n)\in B(H)^n$. Let $k\in\NN\cup\{\infty\}$. Define the $k$-th rank numerical range of $\mT$ as the set of all $\la=(\la_1,\dots,\la_n)\in\CC^n$ such that there exists a subspace $L\subset H$ with $\dim L=k$ satisfying
$$
P_LT_jP_L=\la_jP_L, \qquad j=1,\dots,n.
$$
The set $W_{\infty}(\mathcal T)$ is called the infinite numerical range of $\mathcal T.$
\end{definition}

Clearly $W_1(\mT)$ is the usual joint numerical range and
$$
W_1(\mT)\supset W_2(\mT)\supset\cdots\supset W_\infty(\mT).
$$
It is easy to see that $W_\infty(\mT)$ can be empty even for $n=1.$ (Consider an injective compact operator $T_1$.)
Using \cite[Corollary 4.2]{MullerT} and the definition of $W_e(\mT)$ it follows that
\begin{equation}\label{einf}
\Int (W_e(\mathcal T)) \subset W_{\infty}(\mathcal T) \subset W_e(\mT)
\end{equation}
 for any $\mathcal T \in B(H)^n$. So $W_\infty(\mT)$ is large whenever $W_e(\mathcal T)$ is large.
On the other hand, in infinite-dimensional spaces the $k$-th rank numerical range is  always nonempty for each $k \in \mathbb N,$
as the following proposition (implicit in \cite{Li11}) shows.

\begin{proposition}\label{non-empty}
Let
$\mT=(T_1,\dots,T_n)\in B(H)^n$. Then $W_k(\mT)\ne\emptyset$ for all $k\in\NN$.
\end{proposition}
\begin{proof}
We prove the statement by induction on $n$.
 By e.g. \cite[Theorem 1]{Li091}, one infers that $W_k(T_1)\ne\emptyset$ for any operator $T_1\in B(K)$ with $\dim K\ge 3k-2$. In particular,
 $W_k(T_1)\ne\emptyset$ for all $k\in\NN$.
Suppose the statement is true for some $n-1\ge 1$. Let $k\in\NN$. By the induction assumption, $W_{4k}(T_1,\dots,T_{n-1})\ne\emptyset$. So there exists $(\la_1,\dots,\la_{n-1})\in \CC^{n-1}$ and a subspace $L\subset H$ with $\dim L= k$ such that $$P_LT_jP_L=\la_j P_L, \qquad j=1,\dots,n-1.$$ By the same result from \cite{Li091}, $W_k(P_LT_nP_L)\ne\emptyset$. So there exists $\la_n\in\CC$ and a subspace $L'\subset L$ with $\dim L'=k$ and $P_{L'}T_nP_{L'}=\la_n P_{L'}$.  Hence $P_{L'}T_iP_{L'}=\la_i P_{L'}, 1 \le i \le n,$ so that $(\la_1,\dots,\la_{n})\in W_k(T_1,\dots,T_n)$ and $W_k(T_1,\dots,T_n)\ne\emptyset$.
\end{proof}

For $k<\infty$ the higher rank numerical ranges $W_k(\mT)$  are, in general, not convex. However, they are always star-shaped, as we prove below.
See \cite[Proposition 4.1]{Li11} for an analogous statement.

\begin{theorem}\label{star}
Let
$\mathcal T=(T_1,\dots,T_n)\in B(H)^n$. Then for every $k \in \mathbb N$ the set
$W_k(\mT)$  is star-shaped with star centers taken from $W_m(\mT)$ for any
$m>k(2n+1).$
\end{theorem}
\begin{proof}
Let $k\in\NN$.
Fix $m >k(2n+1).$
By Proposition \ref{non-empty},
it follows that $W_m(\mT)\ne\emptyset$, so we can choose $\la\in W_m(\mT)\subset W_k(\mT)$. We show that $W_k(\mT)$ is star-shaped with the center $\la$.

Let $\mu\in W_k(\mT)$ and $t\in[0,1]$. There exists a subspace $L\subset H$ with $\dim L=k$ and $P_L\mT P_L=\mu P_L$.
Let $x_1,\dots,x_k$ be an orthonormal set in $L$.
Let $M\subset H$ satisfy $\dim M=m$ and $P_M\mT P_M=\la P_M$.
We construct an orthonormal set $y_1,\dots,y_k\in M$ in the following way:
Let $y_1$ be any unit vector in $M\cap\{L, T_jL, T_j^*L:1\le j\le n\}^\perp$. Choose inductively unit vectors $y_s\in M, 2\le s \le k,$ such that
\begin{align*}
y_s &\perp \{ L, T_jL, T^*_jL: \quad 1\le j\le n,\} \\
y_s &\perp\{y_i,T_j y_i,T_j^*y_i:\quad 1\le i\le s-1, 1\le j\le n\}.\
\end{align*}
Let
$$
u_s:=\sqrt{t}x_s+\sqrt{1-t}y_s, \quad s=1,\dots,k \qquad \text{and} \qquad  L':=\bigvee_{s=1}^ku_s.
 $$
 Clearly $\dim L'=k$ and the vectors $u_1,\dots,u_k$ form an orthonormal basis in $L'$. If $y\in L', \|y\|=1$, then $y=\sum_{s=1}^k\al_su_s$ for some $\{\al_s: 1 \le s \le k\}\subset \mathbb C$ with $\sum_{s=1}^k|\al_s|^2=1$. We have
\begin{align*}
\langle \mT y,y\rangle=&
\Bigl\langle \mT\sum_{s=1}^k\al_s\sqrt{t}x_s,\sum_{s=1}^k\al_s\sqrt{t}x_s\Bigr\rangle
+
2\Re \Bigl\langle \mT\sum_{s=1}^k\al_s\sqrt{t}x_s,\sum_{s=1}^k\al_s\sqrt{1-t}y_s\Bigr\rangle\\
+&
\Bigl\langle \mT\sum_{s=1}^k\al_s\sqrt{1-t}y_s,\sum_{s=1}^k\al_s\sqrt{1-t}y_s\Bigr\rangle\\
=&\mu\Bigl\|\sum_{s=1}^k\al_s\sqrt{t}x_s\Bigr\|^2+
\la\Bigl\|\sum_{s=1}^k\al_s\sqrt{1-t}x_s\Bigr\|^2
\\
=&
t\mu+(1-t)\la.
\end{align*}
Hence $t\mu+(1-t)\la\in W_k(\mT),$ so the set $W_k(\mT)$ is star-shaped with the center at $\la \in W_m(\mT),$ as required.
\end{proof}

\begin{remark} It is easy to see that the closure $\overline{W_k(\mT)}$ is also star-shaped.
\end{remark}

The infinite and essential numerical range  have ``infinite-dimensional'' nature. However
it is possible to describe them in terms of ``finite-dimensional'' higher rank numerical ranges.
Moreover, we  characterize  the infinite and essential numerical ranges of tuples by means of compressions
of tuples to infinite-dimensional subspaces.

The equivalence (i)$\Leftrightarrow$(ii) in the proposition below was fist proved in \cite[Theorem 4.1]{Li11}.
\begin{proposition}\label{infinite}
Let $\mathcal T=(T_1,\dots,T_n)\in B(H)^n$ and $\la=(\la_1,\dots,\la_n)\in\CC^n$. The following statements are equivalent:
\begin{itemize}
\item [(i)] $\la\in W_\infty(\mT)$;

\item [(ii)] $\la\in\bigcap_{k=1}^\infty W_k(\mT)$;

\item [(iii)] for every subspace $M\subset H$ of finite codimension there exists a unit vector $x\in M$ such that $\langle\mT x,x\rangle=\la$.
\end{itemize}
\end{proposition}
\begin{proof}
The implication
(i)$\Rightarrow$(ii) is clear.
\smallskip

(ii)$\Rightarrow$(iii):
Let $M\subset H$ be a subspace of finite codimension. Let $k\in\NN$, $k>\codim M$. By  (ii), there exists a subspace $F\subset H$ with $\dim F=k$ and $P_F\mT P_F=\la P_F$. Then $F\cap M\ne\{0\},$ and any unit vector in $F\cap M$ satisfies {\it (iii)}.
\smallskip

(iii)$\Rightarrow$(i):
Using (iii), find a unit vector $x_1\in H$ such that $\langle\mathcal T x_1,x_1\rangle=\la$. Construct inductively a sequence $(x_i)_{i \ge 1}\subset H$ of unit vectors such that
$$
x_{i+1}\perp\bigl\{x_{m}, T_jx_m,T_j^{*}x_m: 1\le m\le i, 1\le j\le n\bigr\}
$$
and
$$
\langle\mathcal T x_i,x_i\rangle=\la
$$
for all $i\in\NN$.
Let $L=\bigvee_{i=1}^\infty x_i$. Clearly $L$ is an infinite-dimensional subspace with an orthonormal basis $(x_i)_{i \ge 1}$.
Let $y\in L$, $\|y\|=1$, so that $y=\sum_{i=1}^\infty\al_i x_i$ where $\sum_{i=1}^\infty|\al_i|^2=1.$
Then
$$
\langle \mT y, y\rangle=
\sum_{i=1}^\infty |\al_i|^2 \langle \mT x_i,x_i\rangle=\la\sum_{i=1}^\infty|\al_i|^2=\la,
$$
so that $P_L\mathcal  T P_L=\mu P_L$.
\end{proof}

Incidentally, in the general setting of operator tuples, Proposition \ref{infinite} gives a partial answer to an old question of Fillmore, Stampfli and Pearcy \cite[p. 190, Remark (4)]{Fillmore} on the description for $T \in B(H)$ of the set of $\la \in \mathbb C$ such that $P(T-\la)P=0$ for an infinite-rank projection $P.$ It can also be considered as a sharper version of \cite[Theorem 3.1.1]{Binding} where
an approximate version of the proposition has been proved.

The equivalence (i)$\Leftrightarrow$(ii) in the next result has been first obtained in \cite[Corollary 4.5]{Li11}.
\begin{proposition}\label{essential}
Let $\mathcal T=(T_1,\dots,T_n)\in B(H)^n$ and $\la=(\la_1,\dots,\la_n)\in\CC^n$. The following statements are equivalent:
\begin{itemize}
\item [(i)] $\la\in W_e(\mT)$;

\item [(ii)] $\la\in\bigcap_{k=1}^\infty \overline{W_k(\mT)}$;

\item [(iii)] for every $\de>0$ and every  subspace $M\subset H$ of finite codimension
there exists a unit vector $x\in M$ such that $\|\langle\mT x,x\rangle-\la\|<\de$.
\end{itemize}
\end{proposition}
\begin{proof}
(iii)$\Rightarrow$(i) is clear.
\smallskip

(ii)$\Rightarrow$(iii):
Let $M\subset H$ be a subspace of a finite codimension and $\de>0$. Let $k\in\NN$, $k>\codim M$. By \emph{(ii)}, there exists
$\mu\in W_k(\mT)$ such that $\|\la-\mu\|<\de$.
Let $F$ be a subspace of $H$ with $\dim F=k$ and $P_F\mT P_F=\mu P_F$. Then $F\cap M\ne\{0\}$. Let $x\in F\cap M$ be any unit vector.  Then
$$
\|\langle\mT x,x\rangle-\la\|=
\|\mu-\la\|<\de.
$$
\smallskip

(i)$\Rightarrow$(iii): This is proved in Lemma \ref{compres}.
\smallskip

(iii)$\Rightarrow$(ii): Let $\la \in \mathbb C^n$ satisfy \emph{(iii)} for some $\de>0.$ Let $k\in\NN$ be fixed. Choose inductively an orthonormal sequence $(x_i)_{i \ge 1}\subset H$ such that
$$
x_{i+1}\perp \bigl\{x_m, T_jx_m, T_j^*x_m: 1\le m\le i, 1\le j\le n\bigr\}
$$
and
$$
\bigl\|\langle \mT x_{i+1},x_{i+1}\rangle-\la\bigr\|<\de.
$$
Let $L=\bigvee_{i=1}^\infty x_i$. Then $\dim L=\infty$ and $\|\langle\mT y,y\rangle-\la\|<\de$ for all $y\in L$, $\|y\|=1$.
Hence $$\bigl\|P_LT_jP_L-\la_jP_L\bigr\|<2\de,\qquad j=1,\dots,n.$$
Note that $W_k(P_L\mT P_L)\ne\emptyset$, and let $\mu\in W_k(P_L\mT P_L)$. Then $\|\mu-\la\|<2\de$.
Since $\de>0$ was arbitrary, $\la\in\overline{W_k(\mT)}$.
\end{proof}

One important application of Propositions \ref{infinite},  (iii) and \ref{essential}, (iii) is the proof of convexity for $W_\infty(\mT)$ and $W_e(\mT).$
By different arguments, the convexity of $W_e(\mathcal T)$ was first proved in \cite[Lemma 3.1]{Bercov} (see also \cite[Theorem 3.1]{Li09}), while the fact that  $W_\infty(\mT)$ is convex was discovered in \cite[Theorem 4.2]{Li11}, see also \cite{Roldan}.
\begin{theorem}
Let
$\mathcal T=(T_1,\dots,T_n)\in B(H)^n$. Then the sets $W_\infty(\mT)$ and $W_e(\mT)$ are convex.
\end{theorem}
\begin{proof}
Let $\la,\mu\in W_\infty(\mT)$ and $t\in[0,1]$. Let $M\subset H$, $\codim M<\infty$. By Proposition \ref{infinite}, there exists $x\in M$, $\|x\|=1$ and $\langle \mT x,x\rangle=\la$.
Similarly there exists a unit vector $y\in M\cap\{x,T_jx,T_j^*x:1\le j\le n\}^\perp$ such that $\langle\mT y,y\rangle=\mu$.
Let $u=\sqrt{t}x+\sqrt{1-t}y$. Clearly $u\in M$, $\|u\|=1$ and
$$
\langle \mT u,u\rangle=t\langle\mT x,x\rangle+(1-t)\langle\mT y,y\rangle=t\la+(1-t)\mu.
$$
By Proposition \ref{infinite} again, $t\la+(1-t)\mu\in W_\infty(\mT)$.

The convexity of $W_e(\mT)$ can be proved similarly using Proposition \ref{essential} instead of Proposition \ref{infinite}.
\end{proof}

Clearly $W_e(\mT)$ is stable under compact perturbations. The behaviour of $W_\infty(\mT)$ under compact perturbations
is described in Theorem
\ref{infiniterange} below. To prove it, we need the following result of independent interest.
\begin{proposition}\label{propinf}
Let $\mT=(T_1,\dots,T_n)\in B(H)^n$, and let $\La\subset W_e(\mT)$ be a countable set. Then there exists an $n$-tuple $\mK=(K_1,\dots,K_n)$ of trace-class normal operators on $H$ such that
$$
\La\subset W_\infty(\mT-\mK).
$$
\end{proposition}
\begin{proof}
Let $\La=\{\la_1,\la_2,\dots\}$. Let $f:\NN\to\NN\times\NN$ be a bijection. For $s\in\NN$ write $f(s)= (f_1(s),f_2(s))$.
We construct inductively an orthonormal sequence $(e_s)_{s\ge 1}\subset H$ in the following way:
Choose a unit vector $e_1$ arbitrarily, fix $s\ge 2$ and suppose that the vectors $e_1,\dots,e_{s-1}\in H$ have already been constructed. Since $\la_{f_1(s)}\in W_e(\mT)$, there exists an orthonormal sequence $(x_k)_{k \ge 1}\subset H$ such that $\lim_{k\to\infty}\langle \mT x_k,x_k\rangle=\la_{f_1(s)}$.
Let $$F_{s}=\bigvee\{e_i,T_je_i,T_j^*e_i:1\le i\le s-1,1\le j\le n\}.$$
 Since $\dim F_s<\infty$, there exists $m\in\NN$ such that
$$
\|P_{F_s}x_m\|\le 2^{-s} \qquad \text{and} \qquad \|\langle \mT x_m,x_m\rangle-\la_{f_1(s)}\|\le 2^{-s}.
$$
Set $e_s=\frac{(I-P_{F_s})x_m}{\|(I-P_{F_s})x_m\|}$. Then $\|e_s\|=1$ and $e_s \perp F_s.$
We also have
\begin{align*}
\|x_m-e_s\|\le&
\bigl\|x_m-(I-P_{F_s})x_m\bigl\|+ \Bigl\|(I-P_{F_s})x_m-\frac{(I-P_{F_s})x_m}{\|(I-P_{F_s})x_m\|}\Bigr\|\\
\le&
\|P_{F_s}x_m\|+\Bigl(1-\frac{1}{1-2^{-s}}\Bigr)\\
\le& 3\cdot 2^{-s}.
\end{align*}

Moreover, if $\e_s=(\e_{s,1},\dots,\e_{s,n})\in\CC^n$ is given by $\e_s=\langle\mT e_s,e_s\rangle-\la_{f_1(s)}$
then
\begin{align*}
|\e_{s,j}|=&
|\langle T_je_s,e_s\rangle-\la_{f_1(s),j}|\\
\le&
|\langle T_j(e_s-x_m),e_s\rangle|+|\langle T_jx_m,e_s-x_m\rangle|+
|\langle T_j x_m,x_m\rangle-\la_{f_1(s),j}|\\
\le&
3\cdot 2^{-s}\|T_j\|+3\cdot 2^{-s}\|T_j\|+2^{-s}, \qquad 1 \le j \le n.
\end{align*}
Now, for every  $j=1,\dots,n$ define $K_j\in B(H)$ by
$$K_j=\sum_{s=1}^\infty \e_{s,j}e_s\otimes e_s.$$ By construction, $K_j$ is a trace-class normal operator and $\langle (T_j-K_j)e_s,e_s\rangle=\la_{f_1(s),j}$ for all $s\in\NN$.

For $k\in\NN$ let $L_k=\bigvee\{e_s: f_1(s)=k\}$. Clearly $\dim L_k=\infty$.
If $f_1(s)=k$ then $\langle T_je_s,e_s\rangle=\la_{k,j}$. Moreover, if $f_1(s)=k=f_1(s')$ and $s\ne s'$ then $\langle T_je_s,e_{s'}\rangle=0=\langle T_je_{s'},e_s\rangle$ for all $j=1,\dots,n$. It is easy to see that this means that
$$
P_{L_k}(\mT-\mK)P_{L_k}=\la_kP_{L_k}.
$$
Hence $\la_k\in W_\infty(\mT-\mK)$.
\end{proof}

Using Proposition \ref{propinf}, we can now express $W_e(\mathcal T)$ in terms of the infinite numerical ranges of
compact perturbations of $W_{\infty}(\mathcal T).$ The result below seems to be new even for single operators.
\begin{theorem}\label{infiniterange}
Let $\mT\in B(H)^n$. Then:
\begin{itemize}
\item [ (i)] $W_e(\mT)=\bigcup_{\mK\in\mK(H)^n} W_\infty(\mT-\mK)$;

\item [(ii)] there exists an $n$-tuple $\mK$ of compact operators such that
$W_e(\mT)=\overline{W_\infty(\mT-\mK)}$.
\end{itemize}
\end{theorem}
\begin{proof}
To show  (i), observe that  $W_e(\mT)=W_e(\mT-\mK)$ and $W_\infty(\mT-K)\subset W_e(\mT)$ for any $n$-tuple of compact operators $\mK.$ Hence
we have the inclusion ``$\supset$''. The other inclusion is clear by Proposition \ref{propinf}.

To prove  (ii), it suffices to apply Proposition \ref{propinf} with $\Lambda$ being any dense countable set
in $W_e(T)$ and to use once again that $W_e(\mT)$ is invariant under compact perturbations.
Since
$$
W_e(\mT)=\overline {\Lambda} \subset \overline{W_{\infty}(\mT-\mK)} \subset W_e(\mT),
$$
the assertion follows.
\end{proof}
Theorem \ref{infiniterange} is a counterpart of \cite[Corollary 13]{MullerSt} where it was proved
that for any $\mT \in B(H)^n$ there exists an $n$-tuple of compact  operators $\mK$
such that $W_e(T)=\overline{W(\mT-\mK)}.$
Note  that the theorem provides a one more proof of convexity of $W_e(\mathcal T)$ once the convexity of $W_{\infty}(\mT)$ is established.

The notion of the infinite numerical range allows us to prove an inclusion result for numerical ranges which complements  Theorem \ref{numranges} and partially generalizes \eqref{we}. (Note however that its proof uses \eqref{we} essentially.)
Let us first remark that
if $V\subset\CC$ is a convex set, then $\Int (\overline{V}) \subset V$.
Indeed, let $\la\in \Int\overline{V}$. We show that $\la\in V$.
Without loss of generality we can assume that $\la=0$. Since $0\in\Int\overline{V}$, there exists $r>0$ such that $\{a\in\CC: |a|\le r\}\subset\overline{V}$. In particular, $a_0:=r$, $a_1:=r\eta$ and $a_2:=r\eta^2$ are elements of $\overline{V}$, where $\eta=e^{2\pi i/3}$.
Now if the elements $b_0, b_1,$ and $ b_2$ belong to $V$, and are sufficiently close to $a_0, a_1$ and $a_2$, respectively, then $0\in\conv\{b_0,b_1,b_2\}$. Since $V$ is convex, we have $0\in V$.

Thus by convexity of $W(T)$ for any $T \in B(H),$ we infer that
\begin{equation}\label{quasi}
\Int \overline{W(T)}\subset W(T).
\end{equation}

While for tuples $\mT \in B(H)^n$ the set $W(\mT)$ is in general not convex, the property \eqref{quasi} nevertheless holds also for $\mT$
if $\Int W_e(\mT)\ne\emptyset$.

\begin{theorem}
Let $\mT=(T_1,\dots,T_n)\in B(H)^n$, $\Int \, (W_e(\mT))\ne\emptyset$. Then
$\Int \overline{W(\mT)}\subset W(\mT)$.
\end{theorem}
\begin{proof}
Let $\Int (W_e(\mT) \ne\emptyset$. Without loss of generality we may assume that $(0,\dots,0)\in \Int W_e(\mT)$.
Let $r>0$ satisfy
$$
\{(z=(z_1,\dots,z_n):\|z\|\le r\}\subset \Int\, (W_e(\mT)).
$$
Let $\la\in\Int\overline{W(\mT)}$, so there exists $s>0$ such that $(1+s)\la\in\Int \, (\overline{W(\mT)})$.
Let $0<\de<sr$. Then there is $\mu\in W(\mT)$ such that $\|\mu-(1+s)\la\|<\de$.
Set $\eta=\la+\frac{\la-\mu}{s}$.
Then
$$
\|\eta\|\le \left \|\la+\frac{\la-(1+s)\la}{s}\right \|+\frac{\de}{s}<r.
$$
So, by \eqref{einf}, $\eta\in\Int \, (W_e(\mT))\subset W_\infty (\mT)$.
Furthermore,
$$
\frac{1}{1+s} \mu+\frac{s}{1+s} \eta=
\frac{1}{1+s}\mu+\frac{s}{1+s}\la+\frac{\la-\mu}{1+s}=\la.
$$
By Theorem \ref{star}, $\la\in W(\mT)$.
\end{proof}

Since the interior of the essential numerical range played an important role above, it is natural
to realize  when it is non-empty. The following simple proposition clarifies the situation
in algebraic terms.
Let $\mathcal S$ stand for the real linear subspace of $B(H)$ formed by the sums of  selfadjoint compact operators on $H$ and real scalar multiples of the identity.

\begin{proposition}
Let $\mT=(T_1,\dots,T_n)\in B(H)^n$. The following statements are equivalent:
\begin{itemize}
\item [(i)] $\Int \, W_e(\mT)\ne\emptyset$;

\item [ii)] $\Int \, W_\infty(\mT)\ne\emptyset$;

\item [(iii)] the operators $\Re \, T_1,\Im \, T_1,\dots,\Re \, T_n,\Im \, T_n$ are linearly independent in the real vector space of all selfadjoint operators modulo $\mathcal S$.
\end{itemize}
More precisely, if $c, t_1,\dots,t_{2n}$ are real numbers such that $\sum_{j=1}^n(t_{2j-1}\Re \, T_j+ t_{2j}\Im \, T_j)-cI$ is compact, then $t_1=\cdots=t_{2n}=0.$
\end{proposition}
\begin{proof}
The equivalence (i)$\Leftrightarrow$(ii) was proved above. Thus, it is enough to prove the equivalence (i)$\Leftrightarrow$ (iii).
 Recall that $W_e(\mT)$ is a nonempty convex set. We consider $W_e(\mT)$ to be a subset of $\RR^{2n}$. So $\Int W_e(\mT)=\emptyset$ if and only if $W_e(\mT)$ is contained in a proper hyperplane in $\RR^{2n}$.
This is equivalent to the existence of a non-trivial $(2n)$-tuple $(t_1,\dots,t_{2n})\in\RR^{2n}$ and $c\in\RR$ such that
$$
\sum_{j=1}^{2n}t_jz_j-c=0
$$ for all $(z_1,\dots,z_{2n})\in W_e(\mT)\subset\RR^{2n}.$
This means that
$$
W_e\Bigl(\sum_{j=1}^n (t_{2j-1}\Re \, T_j+t_{2j}\Im \, T_j)-cI\Bigr)\subset\{0\},
$$
i.e.,
$$\sum_{j=1}^n (t_{2j-1}\Re \, T_j+t_{2j}\Im \, T_j) -cI$$ is a compact operator.
\end{proof}

\section{Asymptotics  of compressions for operator iterates}\label{restrict}

In this section the numerical ranges ideology will be used to study asymptotical properties of powers of bounded operators.
We will  show that if the powers of $T \in B(H)$  vanish in the \emph{weak} operator topology and the spectrum of $T$ is large enough, then for any strict contraction $C$ it is possible to find a subspace $L\subset H$ such that the compressions of $(T^n)_L$ to $L$ match asymptotically the powers of $C$ in the \emph{uniform} operator topology. Moreover,  if the  assumption $T^n\to 0$ in the weak operator topology  is dropped
then for each $k\in\NN$ we are able to construct a subspace $L\subset H$ such that $(T^n)_L=C_u^n, 1 \le n \le k$, where $C_u\in B(L)$ is a  contraction unitarily equivalent to $C$.

In the rest of this section we fix a separable infinite-dimensional Hilbert space $H$.
We first recall several additional notions and notations for operator tuples needed for the sequel.
 Let $\mA_j\in B(H)^n, 1 \le j \le r,$ so that $\mA_j=(A_{j1},\dots,A_{jn})$ for every $j.$
 The direct sum  $\bigoplus_{j=1}^r\mA_j$ is then defined as the $n$-tuple
$$\Bigl(\bigoplus_{j=1}^r A_{j1},\dots,\bigoplus_{j=1}^r A_{jn}\bigr) \in B\Bigl(\bigoplus_{j=1}^r H\Bigr)^n.$$
Note that if $M$ is a subspace of a Hilbert space $H,$ and  $T_M:M\to M$ is the compression $P_M T P_M$ of $T \in B(H)$ to $M$,
then $T_M=J_M^*TJ_M,$ where the natural embedding $J_M:M\to H$ is defined by $Jx=x, x \in M$.
So, if  $M\subset \bigoplus_{j=1}^r H$ and $\mA_j\in B(H)^n, 1 \le j \le r,$, then we define the compression $\Bigl(\bigoplus_{j=1}^r\mA_j\Bigr)_M$ as the $n$-tuple
$$
J_M^*\Bigl(\bigoplus_{j=1}^r\mA_j\Bigr)J_M=
\Bigl(J_M^*\Bigl(\bigoplus_{j=1}^rA_{j1}\Bigr)J_M,\dots,
J_M^*\Bigl(\bigoplus_{j=1}^rA_{jn}\Bigr)J_M\Bigr)\in B(M)^n.
$$

The next statement, of interest in itself, is an extension of \cite[Proposition 1.1]{PokrzywaJOT} from the case of a single operator to the case
of operator tuples. It will allow us to identify a convex combination of operator tuples with a compression of their direct sum.
\begin{lemma}\label{pokrzywa}
Let $n,r\in\NN$, and let $\mA_j\in B(H)^n,  \quad j=1,\dots,r$. Let $\al_1,\dots,\al_r\ge 0$, $\sum_{j=1}^r\al_j=1$. Then there exists
a subspace $M\subset \bigoplus _{j=1}^r H$ such that
$$ \Bigl(\bigoplus_{j=1}^r \mA_j\Bigr)_M\ \usim\  \sum_{j=1}^r\al_j\mA_j.$$
\end{lemma}

\begin{proof}
For $r=1$ the statement is trivial.
We prove the statement first for $r=2$.

Consider the operator $\UU:H\oplus H\to H\oplus H$ defined by
$$
\UU=\begin{pmatrix}
\sqrt{\al_1}&\sqrt{\al_2}\cr
\sqrt{\al_2}&-\sqrt{\al_1}\end{pmatrix}.
$$
It is easy to verify that $\UU^*=\UU=\UU^{-1}$. Moreover,
$$
\UU\begin{pmatrix}
\mA_1&0\cr
0&\mA_2
\end{pmatrix}\UU=
\begin{pmatrix}
\al_1\mA_1+\al_2\mA_2 &*\cr
*&*
\end{pmatrix}.
$$
Let $M=\UU(H\oplus\{0\})$. Let $J_M:M\to H\oplus H$ be the natural embedding. Then
$$
J_MJ_M^*=P_M=\UU\begin{pmatrix}1&0\cr
0&0
\end{pmatrix}
\UU,
$$
where $P_M$ is the orthogonal projection onto $M.$
Let $J_1:H\to H\oplus H$ be defined by $J_1h=h\oplus 0, \quad h\in H$. Then $J_M^*\UU J_1:H\to M$ is a unitary operator.

We have
\begin{align*}
\begin{pmatrix}
\mA_1&0\\
0&\mA_2
\end{pmatrix}_M&=
J_M^*\begin{pmatrix}
\mA_1&0\cr
0&\mA_2
\end{pmatrix}
J_M\\
&\usim\
(J_1^*\UU J_M)J_M^*
\begin{pmatrix}
\mA_1&0\cr
0&\mA_2\end{pmatrix}
J_M(J_M^*\UU J_1)\\
&=
J_1^*\UU P_M \begin{pmatrix}
\mA_1&0\cr
0&\mA_2\end{pmatrix}
P_M\UU J_1\\
&=
J_1^*
\begin{pmatrix}
1&0\cr
0&0\end{pmatrix}
\UU
\begin{pmatrix}
\mA_1&0\cr
0&\mA_2\end{pmatrix}
\UU
\begin{pmatrix}
1&0\cr
0&0\end{pmatrix}
J_1\\
&=J_1^*
\begin{pmatrix}
1&0\cr
0&0\end{pmatrix}
\begin{pmatrix}
\al_1\mA_1+\al_2 \mA_2&*\cr
*&*\end{pmatrix}
\begin{pmatrix}
1&0\cr
0&0\end{pmatrix}
J_1\\
&=
J_1^*
\begin{pmatrix}
\al_1\mA_1+\al_2 \mA_2&0\cr
0&0\end{pmatrix}
J_1\\
&=\al_1\mA_1+\al_2 \mA_2.
\end{align*}

For $r>2 $ the statement can be proved by induction.
Let $r\ge 3$ be fixed and suppose the statement is true for $r-1$. We may assume that $\sum_{j=1}^{r-1}\al_j\ne 0$.
By the induction hypothesis, there exists a subspace $L\subset\bigoplus_{j=1}^{r-1}H$ such that
$$
J_L^*\Bigl(\bigoplus_{j=1}^{r-1}\mA_j\Bigr)J_L
\ \usim\
\Bigl(\sum_{j=1}^{r-1}\al_j\Bigr)^{-1}\sum_{j=1}^{r-1}\al_j\mA_j.
$$
Consider the Hilbert space
$L\oplus H$. By the statement for $r=2,$ there exists a subspace $M\subset L\oplus H\subset\bigoplus_{j=1}^r H$ such that
$$
J_M^*\Bigl(\bigoplus_{j=1}^r\mA_j\Bigr)J_M
\ \usim\
\Bigl(\sum_{j=1}^{r-1}\al_j\Bigr)\cdot\frac{\sum_{j=1}^{r-1}\al_j\mA_j}
{\sum_{j=1}^{r-1}\al_j}
+\al_r\mA_r
=\sum_{j=1}^r\al_j\mA_j,
$$
and the statement is thus true for $r.$ This completes the proof.
\end{proof}

\begin{remark}
Let $\mA_j\in B(H)^n$ and $\al_1,\dots,\al_r$ be as above. Let $\widetilde\mA_j\in B(H)^n$ be $n$-tuples unitarily equivalent to $\mA_j$, i.e., $\widetilde\mA_j=U_j^{-1}\mA_j U_j$ for some unitary operators $U_j\in B(H), j=1,\dots,r$.

Since $\bigoplus_{j=1}^r\mA_j\ \usim\ \bigoplus_{j=1}^r\widetilde\mA_j$, the previous lemma implies also that there exists a subspace $\widetilde L\subset H$ such that
$$\Bigl(\bigoplus_{j=1}^r\mA_j\Bigr)_{\widetilde L}\ \usim\ \sum_{j=1}^r\al_j\widetilde\mA_j.$$
\end{remark}

Let $S\subset\CC^n$. Denote by $\mM(S)$ the set of all $n$-tuples of operators $\mA=(A_1,\dots,A_n)\in B(H)^n$  such that there exist an orthonormal basis $(x_i)_{i \ge 1}$ in $H$ and elements $\la_i\in S, i \ge 1,$ satisfying  $\mA x_i=\la_i x_i, i \ge 1$.

Using Proposition \ref{pokrzywa} we will further identify a compression of a tuple $\mT$ with a tuple of diagonal operators $\mA$
whose diagonals belong to the infinite numerical range of $\mT.$
\begin{proposition}\label{propm}
Let $\mT=(T_1,\dots,T_n)\in B(H)^n$. Let $\mA\in\conv\mM(W_\infty(\mT))$. Then there exists a subspace $L\subset H$ such that the compression $\mT_L$ is unitarily equivalent to $\mA$.
\end{proposition}

\begin{proof}
 By assumption there exists $r \in \mathbb N$ such that
 $$\mA=\sum_{j=1}^r\al_j\mA_j$$ for some $\mA_j\in\mM(W_\infty(\mT)),1 \le j \le r,$ and $(\al_j)_{1 \le j \le k}$ are nonnegative numbers satisfying $\sum_{j=1}^r\al_j=1$. So for every $j$, $1 \le j \le r,$ we have $\mA_j=\diag(\la_{j,1},\la_{j,2},\dots)$ with $(\la_{j,i})_{ i \ge 1}\in W_\infty(\mT)$.

Consider the set $\{1,\dots,r\}\times\NN$ with the lexicographic order:
$(j,i)\prec (j',i')$ if either $i<i'$ or $i=i'$ and $j<j'$.
We construct inductively an orthonormal sequence $(x_{j,i})\subset H, i, j \in \mathbb N,$ in the following way:
Fix a unit vector $x_{1,1} \in H,$ let $j\in\{1,\dots,r\}, i\in\NN$ and suppose that we have already constructed vectors $x_{j',i'}$ for all $(j',i')\prec (j,i)$.
Since $\la_{j,i}\in W_\infty(\mT)$, we can find a unit vector $x_{j,i}\in H$ such that
\begin{align*}
x_{j,i}\perp x_{j',i'},& \qquad (j',i')\prec(j,i),\\
x_{j,i}\perp T_sx_{j',i'},& \qquad (j',i')\prec(j,i),\,\, s=1,\dots,k,\\
x_{j,i}\perp T^*_sx_{j',i'},& \qquad (j',i')\prec(j,i), \,\, s=1,\dots,k,
\end{align*}
and
$$
\langle\mT x_{j,i},x_{j,i}\rangle=\la_{j,i}.
$$
Suppose we have constructed the vectors $x_{j,i}$ in this way. For $j=1,\dots,r$ let $H_j=\bigvee_{i\in\NN} x_{j,i}$. Let $\widetilde H=\bigoplus_{j=1}^r H_j$.
For $j=1,\dots,r$ we have
$$
J_{H_j}^*\mT J_{H_j}\ \usim\ \mA_j
$$
and
$$
J_{\widetilde H}^*\mT J_{\widetilde H}\ \usim\
\bigoplus_{j=1}^r \mA_j.
$$
By Lemma \ref{pokrzywa}, there exists a subspace $L\subset \widetilde H\subset H$ such that
$$J_L^*\mT J_L\ \usim\ \sum_{j=1}^r\al_j\mA_j.$$
\end{proof}

Now we are able to prove a partial generalization of \cite[Theorem 2.1]{Bourin03}
by putting the result into the setting of operator tuples.

\begin{theorem}\label{bourin}
Let $T\in B(H)$ be such that the polynomial hull $\hat\sigma(T)$ contains $\DD$. Let $n\in\NN$ and $\mT=(T,T^2,\dots,T^n)$. Let $A\in B(H)$, $\|A\|<1$ and let $\mA=(A,A^2,\dots,A^n)$. Then there exists a subspace $L\subset H$ such that
$$
(\mT)_L\ \usim\ \mA.
$$
\end{theorem}

\begin{proof}
Let $\|A\|=c<1$. Since $A$ has the (power) dilation $cU$ where $U$ is a unitary operator, we can assume that $A=cU$.

By Theorem \ref{lambdawe} and \eqref{einf}, we have $(\la,\la^2,\dots,\la^n)\in  W_\infty(\mT)$ for all $\la\in\DD$.
Since the function
$$
\la\mapsto \dist\bigl\{(\la,\dots,\la^n),\CC^n\setminus W_\infty(\mT)\bigr\}
$$ is continuous, we can
find $c'$  such that $c<c'<1$ and $\frac{c'}{c}(\la,\la^2,\dots,\la^n)\in W_{\infty}(\mT)$ for all $\la$ with $|\la|=c$.
Set $$\eta=\Bigl(1-\frac{c}{c'}\Bigr)(2k)^{-1}.$$
 Let $\de$ be such that  $(\e_1,\dots,\e_n)\in W_e(\mT)$ for all $\e_1,\dots,\e_n\in\CC$ with $\max_j|\e_j|<\de$.
By the Weyl-von Neumann diagonalization theorem (see e.g. \cite[Chapter 6.37-38]{Conway}),
we can decompose $A$ as $A=D+K_1$, where $D$ is a diagonal operator with entries of modulus $c$ and $K_1$ is a compact operator satisfying $\|K_1\|<\de\eta k^{-1}$.
Set $K_j:=A^j-D^j, 1 \le j \le n$. We have
$$
K_j=A^j-D^j=\sum_{i=0}^{j-1}A^i(A-D)D^{j-i-1}.
$$
So for every $j$ the operator $K_j$ is compact and $$\|K_j\|\le n c^{j-1}\|A-D\|\le n\|K_1\|\le\de\eta.$$
Write $K_j=\Re K_j+ i\,\Im K_j,1 \le j \le n$.
The operators $\Re K_j$ and $i\,\Im K_j$ are diagonal operators with entries of modulus at most $\de\eta$.
We have
\begin{align*}
(A,A^2,\dots,A^n)&=
\frac{c}{c'}\cdot\frac{c'}{c}(D,D^2,\dots,D^n)+
\eta\cdot\eta^{-1}(\Re K_1,0,\dots,0)\\
&+\eta\cdot\eta^{-1}(i\,\Im K_1,0,\dots,0)+
\eta\cdot\eta^{-1}(0,\Re K_2,0,\dots,0)+\cdots\\
&\cdots+
\eta\cdot\eta^{-1}(0,\dots,0,i\,\Im  K_n),
\end{align*}
where
$\frac{c}{c'}+2k\eta=1$ and the $n$-tuples
$$
\frac{c'}{c}(D,D^2,\dots,D^n),\,\,
\eta^{-1}(\Re K_1,0,\dots,0),\dots,\,\,
\eta^{-1}(0,\dots,0,i\,\Im K_n)
$$
belong to $\mM(W_\infty(\mT))$. By Proposition \ref{propm}, there exists a subspace $L\subset H$ such that
$$
(\mT)_L\ \usim \ \mA.
$$
\end{proof}

We proceed with an asymptotic  version of Theorem \ref{bourin}. To this aim several auxiliary lemmas will be needed.
First we prove a slight generalization of \cite[Lemma 6.1]{MullerT} where we replace the element $(0,\dots, 0)$ by $(\la,\la^2,\dots,\la^n), |\la|<1,$
and choose a unit vector $x \in H$ more carefully.
The next statement can be considered as a variant of the main result, Theorem \ref{maint} below, for a single vector.
\begin{lemma}\label{lemmahamdan}
Let $T\in B(H)$ be such that $T^n\to 0$ in the weak operator topology, and let $\la \in \mathbb C, |\la|<1.$ Suppose that
\begin{equation}\label{wen}
(\la,\la^2,\dots,\la^n)\in W_\infty(T,\dots,T^n)
\end{equation}
for all $n\in \NN.$
Let $A\subset H$ be a finite set, $\e>0$ and $M\subset H$ be a subspace of a finite codimension. Then there exists a unit vector $x \in M\cap A^\perp$ such that
$$
\sup_{n\ge 1}|\langle (T^n-\lambda^n) x,x\rangle|\le\e,\quad
\sup_{n\ge 1}|\langle T^nx,a\rangle|\le\e\quad\hbox{and}\quad
\sup_{n\ge 1}|\langle T^{*n}x,a\rangle|\le\e
$$
for all $a\in A.$
\end{lemma}

\begin{proof}
Clearly $T$ is power bounded by the uniform boundedness principle. Let $K=\sup\{\|T^n\|:n=0,1,\dots\}$.
It is apparent that also $T^{*n}\to 0$ in the weak operator topology.
Without loss of generality we may assume that $\max\{\|a\|:a\in A\}\le 1.$

Choose $s\in\NN$ such that $s>25K^2\e^{-2}$, and find $n_0\in\NN$ such that $|\la|^{n_0}<\frac{\e}{5}$.

We construct unit vectors $u_1, u_2,\dots,u_s\in M$ and positive integers $n_0<n_1<\cdots<n_s$ in the following way:
Fix a unit vector $u_1 \in M,$ let $1\le r\le s-1$ and suppose that the unit vectors $u_1,\dots,u_r\in M$ and numbers $n_1<\cdots<n_r$ have already been constructed.

By Proposition \ref{essential}, there exists a unit vector $u_{r+1}\in M$ such that
$$
u_{r+1}\perp\{T^nu_k, T^{*n}u_k,T^na,T^{*n}a:\, 0\le n\le n_r, 1\le k\le r, a\in A\}
$$
and
$$
\langle T^nu_{r+1},u_{r+1}\rangle=\la^n, \qquad 1\le n\le n_r.
$$
Find $n_{r+1}>n_r$ such that
\begin{align*}
|\langle T^nu_{r+1},u_k\rangle|<&\frac{\e}{5s},\\
|\langle T^{*n}u_{r+1},u_k\rangle|<&\frac{\e}{5s},\\
|\langle T^nu_{r+1},a\rangle|<&\frac{\e}{5s},
\end{align*}
and
$$
|\langle T^{*n}u_{r+1},a\rangle|<\frac{\e}{5s}
$$
for all $n\ge n_{r+1}$, $1\le k\le r+1$ and $a\in A$.

Let $u_1,\dots,u_s$ and $n_0,\dots,n_s$ be constructed in this way.
Set $$x=\frac{1}{\sqrt{s}}\sum_{k=1}^s u_k.$$
Clearly $x\in M.$ Moreover, $\|x\|=1$ since the vectors $u_k$ are orthonormal.

For $n> n_s$ we have
$$
\bigl|\langle (T^n-\la^n)x,x\rangle\bigr|\le
s^{-1}\sum_{k,k'=1}^s|\langle T^nu_k,u_{k'}\rangle|+|\la|^n\le
s^{-1}s^2\frac{\e}{5s}+\frac{\e}{5}<\e.
$$

Let $0\le r\le s-1$ and $n_r<n\le n_{r+1}$. Then
\begin{align*}
\bigl|\langle (T^n-\la^n)x,x\rangle\bigr|&\le
|\langle T^n x,x\rangle|+\frac{\e}{5}\\
&\le
s^{-1}\sum_{k,k'=1}^{r}|\langle T^nu_k,u_{k'}\rangle| +
s^{-1}\sum_{k=1}^{r+1}|\langle T^nu_{r+1},u_{k}\rangle|\\
&+s^{-1}\sum_{k=1}^{r}|\langle T^nu_k,u_{r+1}\rangle|
+
s^{-1}\sum_{k=r+2}^{s}|\langle T^n u_k,u_{k}\rangle|\\
&+
s^{-1}\sum_{1\le k,k'\le s, k\ne k'\atop \max\{k,k'\}\ge r+2}|\langle T^nu_k,u_{k'}\rangle|+\frac{\e}{5},
\end{align*}
where the last sum is equal to $0$ by the construction.
So
\begin{align*}
\bigl|\langle (T^n-\la^n)x,x\rangle\bigr|\le&
s^{-1}r^2\frac{\e}{5s}+
s^{-1}\|T^nu_{r+1}\|\cdot\Bigl\|\sum_{k=1}^{r+1}u_k\Bigr\|\\
+&
s^{-1}\|T^{*n}u_{r+1}\|\cdot\Bigl\|\sum_{k=1}^{r}u_k\Bigr\|
+s^{-1}(s-r-1)|\la|^n+\frac{\e}{5}\\
\le&
\frac{\e}{5}
+s^{-1}K\sqrt{r+1}+s^{-1}K\sqrt{r}+\frac{\e}{5}+\frac{\e}{5}\\
\le&\e.
\end{align*}

Let $1\le n\le n_0$. Then
\begin{align*}
&\bigl\langle (T^n-\la^n)x,x\bigr\rangle\\
=&
s^{-1}\sum_{k=1}^s\langle (T^n-\la^n)u_k,u_k\rangle
+s^{-1}\sum_{1\le k,k'\le s, k\ne k'}\langle (T^n-\la^n)u_k,u_{k'}\rangle=0.
\end{align*}

Hence $$\sup_{n\ge 1}|\langle (T^n-\la^n)x,x\rangle|\le \e.$$

Let $a\in A$. For $n\ge n_{s}$ we have
$$
|\langle T^nx,a\rangle|\le
\frac{1}{\sqrt{s}}\sum_{k=1}^s|\langle T^nu_k,a\rangle|\le
\frac{1}{\sqrt{s}}\cdot s\cdot\frac{\e}{5s}<\e.
$$
Let $0\le r\le s-1$ and $n_r\le n<n_{r+1}$. Then
\begin{align*}
|\langle T^nx,a\rangle|&\le
\frac{1}{\sqrt{s}}\sum_{k=1}^r|\langle T^nu_k,a\rangle|
+\frac{1}{\sqrt{s}}|\langle T^nu_{r+1},a\rangle|
+\frac{1}{\sqrt{s}}\sum_{k=r+2}^s|\langle T^nu_k,a\rangle|\\
&\le
\frac{1}{\sqrt{s}}\cdot r\cdot\frac{\e}{5s}+
\frac{1}{\sqrt{s}}\cdot K+0<\e.
\end{align*}
Finally, for $1\le n\le n_0$ we have
$$
\langle T^nx,a\rangle= s^{-1}\sum_{k=1}^s\langle T^nu_k,a\rangle=0.
$$
Thus $$\sup_{n\ge 1}|\langle T^nx,a\rangle|\le\e$$ for all $a\in A$.

The property $\sup_{n\ge 1}|\langle T^{*n}x,a\rangle|\le\e$ for all $a \in A$ can be proved similarly.

\end{proof}

The previous lemma enables us to prove the main result in a particular situation when  the compression of powers a bounded operator is approximated
 by powers
of any strictly contractive diagonal operator.
\begin{lemma}\label{diagonal}
Let $T\in B(H)$ be such that $T^n\to 0$ in the weak operator topology and  $\sigma(T)\supset\mathbb T$. Let $(\la_k)_{k \ge 1}\subset\mathbb D$, $\sup_{k\ge 1}|\la_k|<1,$ and $\e>0$. Then there exists an orthonormal sequence $(e_k)_{k \ge 1}$ in $H$ such that
$$
\sup_{n\ge 1}\|(T^n)_L-D^n\|\le\e
$$
$$
\hbox{and}\quad\lim_{n\to\infty}\|(T^n)_L-D^n\|=0,
$$
where $L=\bigvee_{k=1}^\infty e_k$ and $D\in B(L)$ is the diagonal operator defined by $De_k=\la_k e_k,\quad k\in\NN$.

\end{lemma}

\begin{proof}
Let $r=\sup_k |\la_k|<1$. By Theorem \ref{lambdawe} and \eqref{einf}, $$(\la_k,\la_k^2,\dots,\la_k^n)\in W_\infty(T,T^2,\dots,T^n)$$ for all $n,k\in\NN$.
Let $n_0\in\NN$ satisfy $r^{n_0}<\frac{\e}{16}$.
We construct vectors $e_1,e_2,\dots$ and positive integers $n_1< \cdots$ also inductively.

Fix a unit vector $e_1 \in H,$ suppose that $s\ge 1$ and orthonormal vectors $e_1,\dots,e_s\in H$ and numbers $n_0<n_1<\cdots<n_s$ have already been constructed.
Using Lemma \ref{lemmahamdan}, find a unit vector $e_{s+1}\in H$ such that
\begin{align*}
e_{s+1}&\perp \{T^ne_k, T^{*n}e_k: 0\le n\le n_s, 1\le k\le s\},\\
\sup_{n\ge 1}&\bigl|\langle (T^n-\la_{s+1}^n)e_{s+1},e_{s+1}\rangle\bigr|<\frac{\e}{2^{s+4}(s+1)},\\
\sup_{n\ge 1}&|\langle T^{n}e_{s+1},e_{k}\rangle|<\frac{\e}{2^{s+4}(s+1)},\qquad 1\le k\le s,
\end{align*}
and
$$
\sup_{n\ge 1}|\langle T^{*n}e_{s+1},e_{k}\rangle|<\frac{\e}{2^{s+4}(s+1)},\qquad 1\le k\le s.
$$
Find $n_{s+1}>n_s$ such that
$
r^{n_{s+1}}<\frac{\e}{2^{s+5}}
$
and
$$
|\langle T^n e_k,e_{k'}\rangle|\le\frac{\e}{2^{s+4}(s+1)}, \qquad 1\le k,k'\le s+1, n\ge n_{s+1}.
$$

Let $L=\bigvee_{k=1}^\infty e_k$ and let the diagonal operator $D:L\to L$ be defined by $De_k=\la_ke_k, \quad k\in\NN$.

Let $x\in L$, $\|x\|=1$. Then $x=\sum_{k\ge 1} \al_ke_k$ where $\sum_{k \ge 1} |\al_k|^2=1$.
Note that $\sum_{k=1}^s|\al_k|\le\sqrt s$ for all $s\in\NN$.

For $1\le n\le n_0$ we have
\begin{align*}
|\langle (T^n-D^n)x,x\rangle|\le&
\sum_{k,k'=1}^\infty|\al_k\bar\al_{k'}|\cdot |\langle (T^n-D^n)e_k,e_{k'}\rangle|\\
=&\sum_{k=1}^\infty|\al_k|^2 \cdot |\langle (T^n-\la_k^n)e_k,e_k\rangle|\le
\frac{\e}{2}.
\end{align*}
Thus $\|(T^n)_L-D^n\|\le\e$.

Let $s\ge 0$ and $n_s< n\le n_{s+1}$. Then
\begin{align*}
|\langle (T^n-D^n)x,x\rangle|
\le& |\langle T^nx,x\rangle|+r^n\\
\le&
\Bigl|\sum_{k,k'=1}^\infty\al_k\bar\al_{k'} \langle T^ne_k,e_{k'}\rangle\Bigr|+\frac{\e}{2^{s+4}}\\
\le&\sum_{k,k'=1}^s|\al_k\bar\al_{k'}|\cdot |\langle T^ne_k,e_{k'}\rangle|
+
\sum_{k=1}^{s}|\al_{s+1}\bar\al_{k}|\cdot |\langle T^ne_{s+1},e_{k}\rangle|\\
+&
\sum_{k=1}^{s}|\al_k\bar\al_{s+1}|\cdot |\langle T^ne_k,e_{s+1}\rangle|
+\sum_{k=s+1}^\infty|\al_k|^2 \cdot|\langle T^ne_k,e_k\rangle|\\
+&\sum_{k\ne k',\max\{k,k'\}\ge s+2}|\al_k\bar\al_{k'}|\cdot |\langle T^ne_k,e_{k'}\rangle|+\frac{\e}{2^{s+4}}\\
\le& \frac{s\e}{2^{s+3}s}+
\frac{\e\sqrt{s}}{2^{s+4}(s+1)}
+\frac{\e\sqrt{s}}{2^{s+4}(s+1)}\\
+&\sum_{k=s+1}^\infty|\al_k|^2\cdot\bigl(|\langle (T^n-D^n)e_k,e_k\rangle|+r^n\bigr)+0+\frac{\e}{2^{s+4}}\\
\le&
\frac{\e}{2^{s+3}}+\frac{\e}{2^{s+4}}+\frac{\e}{2^{s+4}}
+\bigl(\frac{\e}{2^{s+4}}+\frac{\e}{2^{s+4}}\bigr)+\frac{\e}{2^{s+4}}\\
<&\frac{\e}{2^{s+1}}.
\end{align*}
Thus
$$
\sup\bigl\{|\langle(T^n-D^n)x,x\rangle|:x\in L, \|x\|=1\bigr\}\le\frac{\e}{2^{s+1}},
$$
and so
$\|(T^n)_L-D^n\|\le \frac{\e}{2^{s}}$.

Hence $$\sup_{n\ge 1}\|(T^n)_L-D^n\|\le\e \qquad \text{and}
\qquad \lim_{n\to\infty}\|(T^n)_L-D^n\|=0.$$
\end{proof}

 Now using dilation theory, we can replace a strictly contractive diagonal operator $D$ in Lemma \ref{diagonal} by any strict contraction unitarily equivalent to a given one.

\begin{theorem}\label{maint}
Let $T\in B(H)$ be such that $T^n\to 0$ in the weak operator topology, and let  $\si(T)\supset\TT$. Let $\tilde C\in B(H)$ and $\|\tilde C\|<1$. Then for every $\e>0$ there exists a subspace $L\subset H$ and $C\in B(L)$ unitarily equivalent to $\tilde C$ such that
$$
\sup_{n \ge 1}\|(T^n)_L- C^n\|\le \e
$$
and
$$
\lim_{n\to\infty}\|(T^n)_L- C^n\|=0.
$$
\end{theorem}

\begin{proof}
Let $\|\tilde C\|<c<1$. Then $\tilde C$ has the power dilation $c\tilde U$ on a Hilbert space $\tilde K$, where $\tilde U\in B(\tilde K)$ is the bilateral shift of infinite multiplicity. So it is sufficient to show the statement for the operator $c\tilde U$.

Find $k\in\NN$ such that $\sup_{n \ge 1} nc^n< \frac{k\e}{4\pi}$.
For $0\le j\le k-1$ let $E_j\subset \tilde K$ be the spectral subspace of $\tilde U$ corresponding to the set $\{e^{2\pi i t}: j/k\le t<(j+1)/k\}$. Consider the operator $\tilde D\in B(\tilde K)$ defined by $\tilde Dx=e^{2\pi ij/k}x, x\in E_j$. Then $\|\tilde U-\tilde D\|\le \frac{2\pi}{k}$ and similarly,
$\|\tilde U^n-\tilde D^n\|\le \frac{2\pi n}{k}$ for all $k$ and $n$ from $\mathbb N.$ Thus
$$
\sup_{n \ge 1} \|(c\tilde U)^n-(c\tilde D)^n\|\le\sup_{n \ge 1} \frac{2\pi nc^n}{k}\le \e/2
$$
and
$$
\lim_{n\to\infty}\|(c\tilde U)^n-(c\tilde D)^n\|=0.
$$
Moreover, $c\tilde D$ is a diagonal operator of the form considered in Lemma \ref{diagonal}. Hence, by Lemma \ref{diagonal},
 there exists a subspace $K\subset H$ and a unitarily equivalent copy  $c D\in B(K)$ of $c\tilde D$ such that
$$
\sup_{n \ge 1}\bigl\|(T^n)_K-(c D)^n\bigr\|\le \e/2
$$
and
$$
\lim_{n\to\infty}\bigl\|(T^n)_K-(c D)^n\bigr\|=0.
$$
Thus
$$
\sup_{n \ge 1}\bigl\|(T^n)_{K}-(cU)^n\bigr\|\le \e
$$
and
$$
\lim_{n\to\infty}\bigl\|(T^n)_{K}-(cU)^n\bigr\|=0,
$$
where $U\in B(K)$ is unitarily equivalent to  $\tilde U$. Hence there exits a subspace $L\subset K$ and a unitarily equivalent  copy
$C\in B(L)$ of $\tilde C$
such that
$$
\sup_{n \ge 1} \|(T^n)_{L}- C^n\|\le \e\qquad\hbox{and}\qquad
\lim_{n\to\infty} \|(T^n)_{L}- C^n\|=0.
$$

\end{proof}

\end{document}